\documentclass{amsart}
\usepackage[all,cmtip]{xy}      % For xy-matrix commutative diagrams
\usepackage{mathtools, amsthm}
\usepackage{amsfonts}
\usepackage{enumitem}  % to change labels in enumerate
\usepackage{setspace}
\usepackage[margin=2.9cm]{geometry}

\newtheorem{theorem}{Theorem}[section]
\newtheorem{lemma}[theorem]{Lemma}
\newtheorem{proposition}[theorem]{Proposition}
\newtheorem{corollary}[theorem]{Corollary}

\newtheorem{remark}[theorem]{Remark}

\begin{document}

\title{A homotopy decomposition of the fibre of the squaring map on $\Omega^3S^{17}$}
\author{Steven Amelotte}

\begin{abstract}
We use Richter's $2$-primary proof of Gray's conjecture to give a homotopy decomposition of the fibre $\Omega^3S^{17}\{2\}$ of the $H$-space squaring map on the triple loop space of the $17$-sphere. This induces a splitting of the mod-$2$ homotopy groups $\pi_\ast(S^{17}; \mathbb{Z}/2\mathbb{Z})$ in terms of the integral homotopy groups of the fibre of the double suspension $E^2:S^{2n-1} \to \Omega^2S^{2n+1}$ and refines a result of Cohen and Selick, who gave similar decompositions for $S^5$ and $S^9$. We relate these decompositions to various Whitehead products in the homotopy groups of mod-$2$ Moore spaces and Stiefel manifolds to show that the Whitehead square $[i_{2n}, i_{2n}]$ of the inclusion of the bottom cell of the Moore space $P^{2n+1}(2)$ is divisible by $2$ if and only if $2n=2, 4, 8$ or $16$.
\end{abstract}

\maketitle

\section{Introduction}

For a based loop space $\Omega X$, let $\Omega X\{k\}$ denote the homotopy fibre of the $k^{th}$ power map $k:\Omega X\to \Omega X$. In \cite{S Odd} and \cite{S Decomposition}, Selick showed that after localizing at an odd prime $p$, there is a homotopy decomposition $\Omega^2S^{2p+1}\{p\} \simeq \Omega^2S^3\langle 3\rangle \times W_p$, where $S^3\langle 3\rangle$ is the $3$-connected cover of $S^3$ and $W_n$ is the homotopy fibre of the double suspension $E^2:S^{2n-1}\to \Omega^2S^{2n+1}$. Since $\Omega^2S^{2p+1}\{p\}$ is homotopy equivalent to the pointed mapping space $\mathrm{Map}_\ast(P^3(p), S^{2p+1})$ and the degree $p$ map on the Moore space $P^3(p)$ is nullhomotopic, an immediate consequence is that $p$ annihilates the $p$-torsion in $\pi_\ast(S^3)$ when $p$ is odd. In \cite{S Reformulation}, Ravenel's solution to the odd primary Arf--Kervaire invariant problem \cite{Ra} was used to show that, at least for $p \ge 5$, similar decompositions of $\Omega^2S^{2n+1}\{p\}$ are not possible if $n \ne 1$ or $p$. 

The $2$-primary analogue of Selick's decomposition, namely that there is a $2$-local homotopy equivalence $\Omega^2S^5\{2\} \simeq \Omega^2S^3\langle 3\rangle \times W_2$, was later proved by Cohen \cite{C 2-primary}. Similarly, since $\Omega^2S^5\{2\}$ is homotopy equivalent to $\mathrm{Map}_\ast(P^3(2), S^5)$ and the degree $4$ map on $P^3(2) \simeq \Sigma\mathbb{R}P^2$ is nullhomotopic, this product decomposition gives a ``geometric" proof of James' classical result that $4$ annihilates the $2$-torsion in $\pi_\ast(S^3)$. Unlike the odd primary case however, for reasons related to the divisibility of the Whitehead square $[\iota_{2n-1},\iota_{2n-1}] \in \pi_{4n-3}(S^{2n-1})$, the fibre of the squaring map on $\Omega^2S^{2n+1}$ admits nontrivial product decompositions for some other values of $n$.

First, in their investigation of the homology of spaces of maps from mod-$2$ Moore spaces to spheres, Campbell, Cohen, Peterson and Selick \cite{CCPS} found that if $2n+1 \ne 3, 5, 9$ or $17$, then $\Omega^2S^{2n+1}\{2\}$ is atomic and hence indecomposable. Following this, it was shown in \cite{CS} that after localization at the prime $2$, there is a homotopy decomposition $\Omega^2S^9\{2\} \simeq BW_2 \times W_4$ and $W_4$ is a retract of $\Omega^3S^{17}\{2\}$. Here $BW_n$ denotes the classifying space of $W_n$ first constructed by Gray \cite{G}. Since $BW_1$ is known to be homotopy equivalent to $\Omega^2S^3\langle 3\rangle$, the pattern suggested by the decompositions of $\Omega^2S^{5}\{2\}$ and $\Omega^2S^{9}\{2\}$ led Cohen and Selick to conjecture that $\Omega^2S^{17}\{2\} \simeq BW_4 \times W_8$. In this note we prove this is true after looping once. (This weaker statement was also conjectured in \cite{C Fibration}.)

\begin{theorem} \label{mainthm}
There is a $2$-local homotopy equivalence $\Omega^3S^{17}\{2\} \simeq W_4\times \Omega W_8$.
\end{theorem}

In addition to the exponent results mentioned above, decompositions of $\Omega^mS^{2n+1}\{p\}$ also give decompositions of homotopy groups of spheres with $\mathbb{Z}/p\mathbb{Z}$ coefficients. Recall that the mod-$p$ homotopy groups of $X$ are defined by $\pi_k(X; \mathbb{Z}/p\mathbb{Z}) = [P^k(p), X]$.

\begin{corollary}
$\pi_k(S^{17}; \mathbb{Z}/2\mathbb{Z}) \cong \pi_{k-4}(W_4) \oplus \pi_{k-3}(W_8)$ for all $k\ge 4$.
\end{corollary}

In section 3 we relate the problem of decomposing $\Omega^2S^{2n+1}\{2\}$ to a problem considered by Mukai and Skopenkov in \cite{MSk} of computing a certain summand in a homotopy group of the mod-$2$ Moore space $P^{2n+1}(2)$---more specifically, the problem of determining when the Whitehead square $[i_{2n}, i_{2n}] \in \pi_{4n-1}(P^{2n+1}(2))$ of the inclusion of the bottom cell $i_{2n}:S^{2n} \to P^{2n+1}(2)$ is divisible by $2$. The indecomposability result for $\Omega^2S^{2n+1}\{2\}$ in \cite{CCPS} (see also \cite{C Course}) was proved by showing that for $n>1$ the existence of a spherical homology class in $H_{4n-3}(\Omega^2S^{2n+1}\{2\})$ imposed by a nontrivial product decomposition implies the existence of an element $\theta \in \pi_{2n-2}^S$ of Kervaire invariant one such that $\theta\eta$ is divisible by $2$, where $\eta$ is the generator of the stable $1$-stem $\pi_1^S$. Such elements are known to exist only for $2n=4, 8$ or $16$. We show that the divisibility of the Whitehead square $[i_{2n},i_{2n}]$ similarly implies the existence of such Kervaire invariant elements to obtain the following.

\begin{theorem}
The Whitehead square $[i_{2n}, i_{2n}] \in \pi_{4n-1}(P^{2n+1}(2))$ is divisible by $2$ if and only if $2n=2,4,8$ or $16$.
\end{theorem}

This will follow from a preliminary result (Proposition \ref{Whitehead products}) equating the divisibility of $[i_{2n}, i_{2n}]$ with the vanishing of a Whitehead product in the mod-$2$ homotopy of the Stiefel manifold $V_{2n+1,2}$, i.e., the unit tangent bundle over $S^{2n}$. It is shown in \cite{S Indecomposability} that there do not exist maps $S^{2n-1}\times P^{2n}(2) \to V_{2n+1,2}$ extending the inclusions of the bottom cell $S^{2n-1}$ and bottom Moore space $P^{2n}(2)$ if $2n \ne 2,4,8$ or $16$. When $2n=2,4$ or $8$, the Whitehead product obstructing an extension is known to vanish for reasons related to Hopf invariant one, leaving only the boundary case $2n=16$ unresolved. We find that the Whitehead product is also trivial in this case.

\section{The decomposition of $\Omega^3S^{17}\{2\}$}

The proof of Theorem \ref{mainthm} will make use of the $2$-primary version of Richter's proof of Gray's conjecture, so we begin by reviewing this conjecture and spelling out some of its consequences. In his construction of a classifying space of the fibre $W_n$ of the double suspension, Gray \cite{G} introduced two $p$-local homotopy fibrations
\[
S^{2n-1} \xrightarrow{\mathmakebox[0.5cm]{E^2}} \Omega^2S^{2n+1} \xrightarrow{\mathmakebox[0.5cm]{\nu}} BW_n
\]

\[
BW_n \xrightarrow{\mathmakebox[0.5cm]{j}} \Omega^2S^{2np+1} \xrightarrow{\mathmakebox[0.5cm]{\phi}} S^{2np-1}
\]
with the property that $j \circ \nu \simeq \Omega H$, where $H: \Omega S^{2n+1} \to \Omega S^{2np+1}$ is the $p^{th}$ James--Hopf invariant. In addition, Gray showed that the composite $BW_n \xrightarrow{j} \Omega^2S^{2np+1} \xrightarrow{p} \Omega^2S^{2np+1}$ is nullhomotopic and conjectured that the composite $\Omega^2S^{2np+1} \xrightarrow{\phi} S^{2np-1} \xrightarrow{E^2} \Omega^2S^{2np+1}$ is homotopic to the $p^{th}$ power map on $\Omega^2S^{2np+1}$. This was recently proved by Richter in \cite{Ri}.

\begin{theorem}[\cite{Ri}] \label{Richter}
For any prime $p$, there is a homotopy fibration \[BW_n \xrightarrow{\mathmakebox[0.5cm]{j}} \Omega^2 S^{2np+1} \xrightarrow{\mathmakebox[0.5cm]{\phi_n}} S^{2np-1}\] such that $E^2 \circ \phi_n \simeq p$.
\end{theorem}

For odd primes, it was shown in \cite{T Anick's} that there is a homotopy fibration $\Omega W_{np} \to BW_n \to \Omega^2S^{2np+1}\{p\}$ based on the fact that a lift $\overline{S}:BW_n \to \Omega^2S^{2np+1}\{p\}$ of $j$ can be chosen to be an $H$-map when $p$ is odd. One consequence of Theorem \ref{Richter} is that this homotopy fibration exists for all primes and can be extended one step to the right by a map $\Omega^2S^{2np+1}\{p\} \to W_{np}$.

\begin{lemma} \label{fibration}
For any prime $p$, there is a homotopy fibration \[BW_n \xrightarrow{\mathmakebox[0.5cm]{}} \Omega^2S^{2np+1}\{p\} \xrightarrow{\mathmakebox[0.5cm]{}} W_{np}.\]
\end{lemma}

\begin{proof}
The homotopy pullback of $\phi_n$ and the fibre inclusion $W_{np}\to S^{2np-1}$ of the double suspension defines a map $\Omega^2S^{2np+1}\{p\} \to W_{np}$ with homotopy fibre $BW_n$, which can be seen by comparing fibres in the homotopy pullback diagram
\begin{equation} \label{diagram2}
\begin{gathered}
\xymatrix{
BW_n \ar[r] \ar@{=}[d] & \Omega^2S^{2np+1}\{p\} \ar[r] \ar[d] & W_{np} \ar[d] \\
BW_n \ar[r]^-j & \Omega^2S^{2np+1} \ar[r]^-{\phi_n} \ar[d]^-p & S^{2np-1} \ar[d]^-{E^2} \\
& \Omega^2S^{2np+1} \ar@{=}[r] & \Omega^2S^{2np+1}.
}
\end{gathered}
\end{equation}
\end{proof}

%Include a remark about the relevance of this fibration to the longstanding conjecture that Anick's space is a double classifying space for W_n?

Looping once, we obtain a homotopy fibration $$W_n \xrightarrow{\mathmakebox[0.5cm]{}} \Omega^3S^{2np+1}\{p\} \xrightarrow{\mathmakebox[0.5cm]{}} \Omega W_{np}$$ which we will show is split when $p=2$ and $n=4$.
We now fix $p=2$ and localize all spaces and maps at the prime $2$. Homology will be taken with mod-$2$ coefficients unless otherwise stated.

The next lemma describes a factorization of the looped second James--Hopf invariant, an odd primary version of which appears in \cite{T Anick's}. By a well-known result due to Barratt, $\Omega H: \Omega^2S^{2n+1} \to \Omega^2S^{4n+1}$ has order $2$ in the group $[\Omega^2S^{2n+1}, \Omega^2S^{4n+1}]$ and hence lifts to a map $\Omega^2S^{2n+1}\to \Omega^2S^{4n+1}\{2\}$. Improving on this, a feature of Richter's construction of the map $\phi_n$ is that the composite $\Omega^2S^{2n+1} \xrightarrow{\Omega H} \Omega^2S^{4n+1} \xrightarrow{\phi_n} S^{4n-1}$ is nullhomotopic \cite[Lemma 4.2]{Ri}. This recovers Gray's fibration $S^{2n-1} \xrightarrow{E^2} \Omega^2S^{2n+1} \xrightarrow{\nu} BW_n$ and the relation $j \circ \nu \simeq \Omega H$ since there then exists a lift $\nu:\Omega^2S^{2n+1} \to BW_n$ making the diagram
\[
\xymatrix{
& BW_n \ar[d]^j \\
\Omega^2S^{2n+1} \ar[r]^-{\Omega H} \ar[ur]^\nu & \Omega^2S^{4n+1}
}
\]
commute up to homotopy. Since $j$ factors through $\Omega^2S^{4n+1}\{2\}$, by composing the lift $\nu$ with the map $BW_n \to \Omega^2S^{4n+1}\{2\}$ we obtain a choice of lift $S:\Omega^2S^{2n+1} \to \Omega^2S^{4n+1}\{2\}$ of the looped James--Hopf invariant. Hence we have the following consequence of Richter's theorem.

\begin{lemma} \label{S}
There is a homotopy commutative diagram
\[
\xymatrix{
\Omega^2S^{2n+1} \ar[d]^\nu \ar[r]^-S & \Omega^2S^{4n+1}\{2\} \ar@{=}[d] \\
BW_n \ar[r] & \Omega^2S^{4n+1}\{2\}
}
\]
where $S$ is a lift of the looped second James--Hopf invariant $\Omega H: \Omega^2S^{2n+1} \to \Omega^2S^{4n+1}$ and the map $BW_n \to \Omega^2S^{4n+1}\{2\}$ has homotopy fibre $\Omega W_{2n}$.
\end{lemma}

The following homological result was proved in \cite{CCPS} and used to obtain the homotopy decompositions of \cite{C 2-primary} and \cite{CS}.

\begin{lemma}[\cite{CCPS}] \label{CCPS}
Let $n\ge 2$ and let $f:X\to \Omega^2S^{2n+1}\{2\}$ be a map which induces an isomorphism on the module of primitives in degrees $2n-2$ and $4n-3$. If the mod-$2$ homology of $X$ is isomorphic to that of $\Omega^2S^{2n+1}\{2\}$ as a coalgebra over the Steenrod algebra, then $f$ is a homology isomorphism.
\end{lemma}

\begin{theorem}
There is a homotopy equivalence $\Omega^3S^{17}\{2\} \simeq W_4\times \Omega W_8$.
\end{theorem}

\begin{proof}
Let $\tau_n$ denote the map $BW_n\to \Omega^2S^{4n+1}\{2\}$ appearing in Lemma \ref{fibration}. By \eqref{diagram2}, $\tau_n$ is a lift of $j$, implying that $\tau_n$ is nonzero in $H_{4n-2}(\;)$ by naturality of the Bockstein since $j$ is nonzero in $H_{4n-1}(\;)$. We can therefore use the maps $\tau_n$ in place of the (potentially different) maps $\sigma_n$ used in \cite{CS} to obtain product decompositions of $\Omega^2S^{4n+1}\{2\}$ for $n=1$ and $2$, the advantage being that $\tau_n$ has fibre $\Omega W_{2n}$. Explicitly, for $n=2$ this is done as follows. By \cite[Corollary 2.1]{CS}, there exists a map $g:\Omega^3S^{17}\{2\} \to \Omega^2S^9\{2\}$ which is nonzero in $H_{13}(\;)$. Letting $\mu$ denote the loop multiplication on $\Omega^2S^9\{2\}$, it follows that the composite
\[
\psi:BW_2\times W_4 \xrightarrow{\mathmakebox[1.5cm]{\tau_2\times (g\circ \Omega \tau_4)}} \Omega^2S^9\{2\}\times\Omega^2S^9\{2\} \xrightarrow{\mathmakebox[0.5cm]{\mu}} \Omega^2S^9\{2\}
\]
induces an isomorphism on the module of primitives in degrees $6$ and $13$. Since $H_\ast(BW_2\times W_4)$ and $H_\ast(\Omega^2S^9\{2\})$ are isomorphic as coalgebras over the Steenrod algebra, the map above is a homology isomorphism by Lemma \ref{CCPS} and hence a homotopy equivalence. 

Now the map $\Omega \tau_4$ fits in the homotopy fibration
\[ W_4 \xrightarrow{\mathmakebox[0.5cm]{\Omega \tau_4}} \Omega^3S^{17}\{2\} \xrightarrow{\mathmakebox[0.5cm]{}} \Omega W_8 \]
and has a left homotopy inverse given by $\pi_2 \circ \psi^{-1}\circ g$ where $\psi^{-1}$ is a homotopy inverse of $\psi$ and $\pi_2:BW_2\times W_4 \to W_4$ is the projection onto the second factor. (Alternatively, composing $g:\Omega^3S^{17}\{2\} \to \Omega^2S^{9}\{2\}$ with the map $\Omega^2S^{9}\{2\} \to W_4$ of Lemma \ref{fibration} yields a left homotopy inverse of $\Omega \tau_4$.) It follows that the homotopy fibration above is fibre homotopy equivalent to the trivial fibration $W_4 \times \Omega W_8 \to \Omega W_8$. 
\end{proof}

\begin{corollary}
$\pi_k(S^{17}; \mathbb{Z}/2\mathbb{Z}) \cong \pi_{k-4}(W_4) \oplus \pi_{k-3}(W_8)$ for all $k\ge 4$.
\end{corollary}

One consequence of the splitting of the fibration $W_n\to \Omega^3S^{4n+1}\{p\} \to \Omega W_{2n}$ when $n\in \{1,2,4\}$ is a corresponding homotopy decomposition of the fibre of the map $S$ appearing in Lemma \ref{S}. As in \cite{T 2-primary}, we define the space $Y$ and the map $t$ by the homotopy fibration
\[ Y\xrightarrow{\mathmakebox[0.5cm]{t}} \Omega^2S^{2n+1}\xrightarrow{\mathmakebox[0.5cm]{S}} \Omega^2S^{4n+1}\{2\}. \]
This space and its odd primary analogue play a central role in the construction of Anick's fibration in \cite{T 2-primary}, \cite{T Anick's} and the alternative proof given in \cite{T A new proof} of Cohen, Moore and Neisendorfer's determination of the odd primary homotopy exponent of spheres. Unlike at odd primes, the lift $S$ of $\Omega H$ cannot be chosen to be an $H$-map. Nevertheless, the corollary below shows that its fibre has the structure of an $H$-space in cases of Hopf invariant one.

\begin{corollary}\label{Y}
There is a homotopy fibration $S^{2n-1} \xrightarrow{f} Y \xrightarrow{g} \Omega W_{2n}$ with the property that the composite $S^{2n-1} \xrightarrow{f} Y \xrightarrow{t} \Omega^2S^{2n+1}$ is homotopic to the double suspension $E^2$. Moreover, if $n=1,2$ or $4$ then the fibration splits, giving a homotopy equivalence
\[ Y \simeq S^{2n-1} \times \Omega W_{2n}. \]
\end{corollary}

\begin{proof}
By Lemma \ref{S}, the homotopy fibration defining $Y$ fits in a homotopy pullback diagram
\[
\xymatrix{
S^{2n-1} \ar[d]^f \ar@{=}[r] & S^{2n-1} \ar[d]^{E^2} \\
Y \ar[d]^g \ar[r]^-t & \Omega^2S^{2n+1} \ar[d]^\nu \ar[r]^-S & \Omega^2S^{4n+1}\{2\} \ar@{=}[d] \\
\Omega W_{2n} \ar[r] & BW_n \ar[r] & \Omega^2S^{4n+1}\{2\},
}
\]
which proves the first statement. Note that when $n=1,2$ or $4$, the map $\Omega W_{2n}\to BW_n$ is nullhomotopic by Theorem \ref{mainthm}, hence $t$ lifts through the double suspension. Since any choice of a lift $Y\to S^{2n-1}$ is degree one in $H_{2n-1}(\;)$, it also serves as a left homotopy inverse of $f$, which implies the asserted splitting.
\end{proof}

\begin{remark}
The first part of Corollary \ref{Y} and an odd primary version are proved by different means in \cite{T 2-primary} and \cite{T A new proof}, respectively (see Remark 6.2 of \cite{T 2-primary}). At odd primes, there is an analogous splitting for $n=1$:
$$Y \simeq S^1 \times \Omega W_p \simeq S^1 \times \Omega^3 T^{2p^2+1}(p)$$
where $T^{2p^2+1}(p)$ is Anick's space (see \cite{T A case}).
\end{remark}

\section{Relations to Whitehead products in Moore spaces and Stiefel manifolds}

The special homotopy decompositions of $\Omega^3S^{2n+1}\{2\}$ discussed in the previous section are made possible by the existence of special elements in the stable homotopy groups of spheres, namely elements of Arf--Kervaire invariant one $\theta \in \pi_{2n-2}^S$ such that $\theta\eta$ is divisible by $2$. In this section, we give several reformulations of the existence of such elements in terms of mod-$2$ Moore spaces and Stiefel manifolds.

Let $i_{n-1}:S^{n-1} \to P^n(2)$ be the inclusion of the bottom cell and let $j_n:P^n(2) \to P^n(2)$ be the identity map. Similarly, let $i'_{2n-1}:S^{2n-1} \to V_{2n+1,2}$ and $j'_{2n}:P^{2n}(2) \to V_{2n+1,2}$ denote the inclusions of the bottom cell and bottom Moore space, respectively.\footnote{Note that we index these maps by the dimension of their source rather than their target, so the element of $\pi_{4n-1}(P^{2n+1}(2))$ we call $[i_{2n}, i_{2n}]$ is called $[i_{2n+1}, i_{2n+1}]$ in \cite{MSk}.}

\begin{proposition} \label{Whitehead products}
The Whitehead product $[i'_{2n-1}, j'_{2n}] \in \pi_{4n-2}(V_{2n+1,2}; \mathbb{Z}/2\mathbb{Z})$ is trivial if and only if the Whitehead square $[i_{2n},i_{2n}] \in \pi_{4n-1}(P^{2n+1}(2))$ is divisible by $2$.
\end{proposition}

\begin{proof}
Let $\lambda:S^{4n-2}\to P^{2n}(2)$ denote the attaching map of the top cell in $V_{2n+1,2} \simeq P^{2n}(2) \cup_\lambda e^{4n-1}$ and note that $[i'_{2n-1}, j'_{2n}] = j'_{2n}\circ [i_{2n-1}, j_{2n}]$ by naturality of the Whitehead product. The map $[i_{2n-1}, j_{2n}]: P^{4n-2}(2) \to P^{2n}(2)$ is essential since its adjoint is a Samelson product with nontrivial Hurewicz image $[u,v] \in H_{4n-3}(\Omega P^{2n}(2))$, where $H_\ast(\Omega P^{2n}(2))$ is isomorphic as an algebra to the tensor algebra $T(u,v)$ with $|u|=2n-2$ and $|v|=2n-1$ by the Bott--Samelson theorem. Since the homotopy fibre of the inclusion $j'_{2n}:P^{2n}(2)\to V_{2n+1,2}$ has $(4n-2)$-skeleton $S^{4n-2}$ which maps into $P^{2n}(2)$ by the attaching map $\lambda$, it follows that $[i'_{2n-1}, j'_{2n}]$ is trivial if and only if $[i_{2n-1}, j_{2n}]$ is homotopic to the composite
\[ P^{4n-2}(2) \xrightarrow{\mathmakebox[0.5cm]{q}} S^{4n-2} \xrightarrow{\mathmakebox[0.5cm]{\lambda}} P^{2n}(2) \]
where $q$ is the pinch map.

To ease notation let $P^n$ denote the mod-$2$ Moore space $P^n(2)$ and consider the morphism of $EHP$ sequences
\[
\xymatrix{
[S^{4n}, P^{2n+1}]\ar[r]^-H \ar[d]^{q^\ast} & [S^{4n}, \Sigma P^{2n}\wedge P^{2n}]\ar[r]^-P\ar[d]^{q^\ast} & [S^{4n-2}, P^{2n}]\ar[r]^-E\ar[d]^{q^\ast} & [S^{4n-1}, P^{2n+1}] \ar[d]^{q^\ast} \\
[P^{4n}, P^{2n+1}]\ar[r]^-H & [P^{4n}, \Sigma P^{2n}\wedge P^{2n}] \ar[r]^-P & [P^{4n-2}, P^{2n}] \ar[r]^-E & [P^{4n-1}, P^{2n+1}] 
}
\]
induced by the pinch map. A homology calculation shows that the $(4n)$-skeleton of $\Sigma P^{2n}\wedge P^{2n}$ is homotopy equivalent to $P^{4n} \vee S^{4n}$. Let $k_1:P^{4n} \to \Sigma P^{2n}\wedge P^{2n}$ and $k_2:S^{4n} \to \Sigma P^{2n}\wedge P^{2n}$ be the composites
\[ P^{4n} \hookrightarrow P^{4n} \vee S^{4n} \simeq \mathrm{sk}_{4n}(\Sigma P^{2n}\wedge P^{2n}) \hookrightarrow \Sigma P^{2n}\wedge P^{2n} \]
and
\[ S^{4n} \hookrightarrow P^{4n} \vee S^{4n} \simeq \mathrm{sk}_{4n}(\Sigma P^{2n}\wedge P^{2n}) \hookrightarrow \Sigma P^{2n}\wedge P^{2n} \]
defined by the left and right wedge summand inclusions, respectively. Then we have that $\pi_{4n}(\Sigma P^{2n}\wedge P^{2n}) = \mathbb{Z}/4\mathbb{Z}\{k_2\}$ and $P(k_2)=\pm 2\lambda$ by \cite[Lemma 12]{M attaching}. 
It follows from the universal coefficient exact sequence
\[ 0 \to \pi_{4n}(\Sigma P^{2n}\wedge P^{2n}) \otimes \mathbb{Z}/2\mathbb{Z} \to \pi_{4n}(\Sigma P^{2n}\wedge P^{2n};\mathbb{Z}/2\mathbb{Z}) \to \mathrm{Tor}(\pi_{4n-1}(\Sigma P^{2n}\wedge P^{2n}), \mathbb{Z}/2\mathbb{Z}) \to 0 \]
that
\begin{align*}
\pi_{4n}(\Sigma P^{2n}\wedge P^{2n}; \mathbb{Z}/2\mathbb{Z}) &= [P^{4n}, \Sigma P^{2n}\wedge P^{2n}] \\
&= \mathbb{Z}/2\mathbb{Z}\{k_1\} \oplus \mathbb{Z}/2\mathbb{Z}\{k_2 \circ q\}
\end{align*}
and that the generator $k_2 \circ q$ is in the kernel of $P$ since $P(k_2)=\pm 2\lambda$ implies $$P(k_2 \circ q)=P(q^\ast (k_2))=q^\ast(P(k_2))= \pm \lambda \circ 2 \circ q=0$$ by the commutativity of the above diagram and the fact that $q:P^{4n-2}\to S^{4n-2}$ and $2:S^{4n-2}\to S^{4n-2}$ are consecutive maps in a cofibration sequence. Therefore $[i_{2n-1}, j_{2n}] = P(k_1)$ since the suspension of a Whitehead product is trivial. On the other hand, $\Sigma \lambda$ is homotopic to the composite $S^{4n-1} \xrightarrow{[\iota_{2n}, \iota_{2n}]} S^{2n} \xrightarrow{i_{2n}} P^{2n+1}$ by \cite{M attaching}, which implies $E(\lambda \circ q) = i_{2n} \circ [\iota_{2n}, \iota_{2n}] \circ q = [i_{2n}, i_{2n}] \circ q$ is trivial in $[P^{4n-1}, P^{2n+1}]$ precisely when $[i_{2n}, i_{2n}]$ is divisible by $2$. Hence $[i_{2n}, i_{2n}]$ is divisible by $2$ if and only if $\lambda \circ q = P(k_1) = [i_{2n-1}, j_{2n}] \in [P^{4n-2}, P^{2n}]$, and the proposition follows.
\end{proof}

We use Proposition \ref{Whitehead products} in two ways. First, since the calculation of $\pi_{31}(P^{17}(2))$ in \cite{MSh} shows that $[i_{16}, i_{16}] = 2\widetilde{\sigma}_{16}^2$ for a suitable choice of representative $\widetilde{\sigma}_{16}^2$ of the Toda bracket $\{\sigma_{16}^2, 2\iota_{16}, i_{16}\}$, it follows that the Whitehead product $[i'_{15}, j'_{16}]: P^{30}(2) \to V_{17,2}$ is nullhomotopic and hence there exists a map $S^{15} \times P^{16}(2) \to V_{17,2}$ extending the wedge of skeletal inclusions $S^{15} \vee P^{16}(2) \to V_{17,2}$. This resolves the only case left unsettled by Theorem 3.2 of \cite{S Indecomposability}. 

In the other direction, note that such maps $S^{2n-1}\times P^{2n}(2) \to V_{2n+1,2}$ restrict to maps $S^{2n-1}\times S^{2n-1}\to V_{2n+1,2}$ which exist only in cases of Kervaire invariant one by \cite[Proposition 2.27]{W}, so Proposition \ref{Whitehead products} shows that when $2n \ne 2^k$ for some $k \ge1$ the Whitehead square $[i_{2n}, i_{2n}]$ cannot be divisible by $2$ for the same reasons that the Whitehead square $[\iota_{2n-1}, \iota_{2n-1}] \in \pi_{4n-3}(S^{2n-1})$ cannot be divisible by $2$. Moreover, since maps $S^{2n-1}\times P^{2n}(2) \to V_{2n+1,2}$ extending the inclusions of $S^{2n-1}$ and $P^{2n}(2)$ are shown not to exist for $2n>16$ in \cite{S Indecomposability}, Proposition \ref{Whitehead products} implies that the Whitehead square $[i_{2n}, i_{2n}]$ is divisible by $2$ if and only if $2n=2,4,8$ or $16$. In all other cases it generates a $\mathbb{Z}/2\mathbb{Z}$ summand in $\pi_{4n-1}(P^{2n+1}(2))$. This improves on the main theorem of \cite{MSk} which shows by other means that $[i_{2n}, i_{2n}]$ is not divisible by $2$ when $2n$ is not a power of $2$.

These results are summarized in Theorem \ref{TFAE} below. First we recall the following well-known equivalent formulations of the Kervaire invariant problem.

\begin{theorem} [\cite{C Course}, \cite{W}]
The following are equivalent:
\begin{enumerate}[label=(\alph*),ref=(\alph*)]
\item The Whitehead square $[\iota_{2n-1}, \iota_{2n-1}] \in \pi_{4n-3}(S^{2n-1})$ is divisible by $2$;
\item There is a map $P^{4n-2}(2) \to \Omega S^{2n}$ which is nonzero in homology;
\item There exists a space $X$ with mod-$2$ cohomology $\widetilde{H}^i(X)\cong \mathbb{Z}/2\mathbb{Z}$ for $i=2n$, $4n-1$, $4n$ and zero otherwise with $Sq^{2n}:H^{2n}(X) \to H^{4n}(X)$ and $Sq^1:H^{4n-1}(X) \to H^{4n}(X)$ isomorphisms;
\item There exists a map $f:S^{2n-1}\times S^{2n-1} \to V_{2n+1,2}$ such that $f|_{S^{2n-1}\times\ast}=f|_{\ast\times S^{2n-1}}$ is the inclusion of the bottom cell;
\item $n =1$ or there exists an element $\theta \in \pi_{2n-2}^S$ of Kervaire invariant one.
\end{enumerate}
\end{theorem}

The above conditions hold for $2n=2, 4, 8, 16, 32$ and $64$, and the recent solution to the Kervaire invariant problem by Hill, Hopkins and Ravenel \cite{HHR} implies that, with the possible exception of $2n=128$, these are the only values for which the conditions hold. Mimicking the reformulations above we obtain the following.

\begin{theorem} \label{TFAE}
The following are equivalent:
\begin{enumerate}[label=(\alph*),ref=(\alph*)]
\item The Whitehead square $[i_{2n}, i_{2n}] \in \pi_{4n-1}(P^{2n+1}(2))$ is divisible by $2$; \label{statement1}
\item There is a map $P^{4n}(2) \to \Omega P^{2n+2}(2)$ which is nonzero in homology; \label{statement2}
\item There exists a space $X$ with mod-$2$ cohomology $\widetilde{H}^i(X)\cong\mathbb{Z}/2\mathbb{Z}$ for $i=2n+1$, $2n+2$, $4n+1$, $4n+2$ and zero otherwise with $Sq^{2n}: H^{2n+1}(X)\to H^{4n+1}(X)$, $Sq^1:H^{2n+1}(X)\to H^{2n+2}(X)$ and $Sq^1:H^{4n+1}(X)\to H^{4n+2}(X)$ isomorphisms; \label{statement3}
\item There exists a map $f: S^{2n-1}\times P^{2n}(2) \to V_{2n+1,2}$ such that $f|_{S^{2n-1}\times\ast}$ and $f|_{\ast\times P^{2n}(2)}$ are the skeletal inclusions of $S^{2n-1}$ and $P^{2n}(2)$, respectively; \label{statement4}
\item $n =1$ or there exists an element $\theta \in \pi_{2n-2}^S$ of Kervaire invariant one such that $\theta\eta$ is divisible by $2$; \label{statement5}
\item $2n=2, 4, 8$ or $16$. \label{statement6}
\end{enumerate}
\end{theorem}

\begin{proof}
\ref{statement1} is equivalent to \ref{statement2}: In the $n=1$ case, $[\iota_2, \iota_2]=2\eta_2$ implies $[i_2, i_2]=0$, and since $\eta_3 \in \pi_4(S^3)$ has order $2$ its adjoint $\widetilde{\eta}_3:S^3 \to \Omega S^3$ extends to a map $P^4(2) \to \Omega S^3$. If this map desuspended, then $\widetilde{\eta}_3$ would be homotopic to a composite $S^3\to P^4(2)\to S^2 \xrightarrow{E} \Omega S^3$, a contradiction since $\pi_3(S^2)\cong \mathbb{Z}$ implies that any map $S^3 \to S^2$ that factors through $P^4(2)$ is nullhomotopic. Hence the map $P^4(2) \to \Omega S^3$ has nontrivial Hopf invariant in $[P^4(2), \Omega S^5]$ from which it follows that $P^4(2) \to \Omega S^3$ is nonzero in $H_4(\;)$. Composing with the inclusion $\Omega S^3\to \Omega P^4(2)$ gives a map $P^4(2) \to \Omega P^4(2)$ which is nonzero in $H_4(\;)$.

Now suppose $n>1$ and $[i_{2n}, i_{2n}] = 2\alpha$ for some $\alpha \in \pi_{4n-1}(P^{2n+1}(2))$. Then $\Sigma\alpha$ has order $2$ so there is an extension $P^{4n+1}(2) \to P^{2n+2}(2)$ whose adjoint $f: P^{4n}(2)\to \Omega P^{2n+2}(2)$ satisfies $f|_{S^{4n-1}} = E \circ \alpha$. Since $\Omega\Sigma(P^{2n+1}(2) \wedge P^{2n+1}(2))$ has $4n$-skeleton $S^{4n}$, to show that $f_\ast$ is nonzero on $H_{4n}(P^{4n}(2))$ it suffices to show that $H_2 \circ f$ is nontrivial in $[P^{4n}(2),  \Omega\Sigma(P^{2n+1}(2) \wedge P^{2n+1}(2))]$ where $H_2: \Omega P^{2n+2}(2) \to \Omega\Sigma(P^{2n+1}(2) \wedge P^{2n+1}(2))$ is the second James--Hopf invariant. If $H_2 \circ f$ is nullhomotopic, then there is a map $g:P^{4n}(2) \to P^{2n+1}(2)$ making the diagram
\[
\xymatrix{
P^{2n+1}(2) \ar[r]^-E  &  \Omega P^{2n+2}(2) \ar[r]^-{H_2}  &  \Omega\Sigma(P^{2n+1}(2) \wedge P^{2n+1}(2)) \\
P^{4n}(2) \ar[u]^g \ar[ur]_f
}
\]
commute. But then $\alpha - g|_{S^{4n-1}}$ is in the kernel of $E_\ast: \pi_{4n-1}(P^{2n+1}(2)) \to \pi_{4n}(P^{2n+2}(2))$ which is generated by $[i_{2n}, i_{2n}]$, so $\alpha - g|_{S^{4n-1}}$ is a multiple of $[i_{2n}, i_{2n}]$. Since $[i_{2n}, i_{2n}]$ has order $2$ and clearly $2g|_{S^{4n-1}}=0$, it follows that $[i_{2n}, i_{2n}]=2\alpha=0$, a contradiction. Therefore $f_\ast$ is nonzero on $H_{4n}(P^{4n}(2))$.

Conversely, assume $n>1$ and $f: P^{4n}(2)\to \Omega P^{2n+2}(2)$ is nonzero in $H_{4n}(\;)$. Since the restriction $f|_{S^{4n-1}}$ lifts through the $(4n-1)$-skeleton of $\Omega P^{2n+2}(2)$, there is a homotopy commutative diagram
\[
\xymatrix{
S^{4n-1} \ar[r] \ar[d]^g  &  P^{4n}(2) \ar[d]^f \\
P^{2n+1}(2) \ar[r]^-E  &  \Omega P^{2n+2}(2)
}
\]
for some map $g:S^{4n-1} \to P^{2n+1}(2)$. Since $E\circ 2g$ is nullhomotopic, $2g$ is a multiple of $[i_{2n}, i_{2n}]$. But if $2g=0$, then $g$ admits an extension $e:P^{4n}(2) \to P^{2n+1}(2)$ and it follows that $f-E\circ e$ factors through the pinch map $q:P^{4n}(2)\to S^{4n}$. This makes the Pontrjagin square $u^2 \in H_{4n}(\Omega P^{2n+2}(2))$ a spherical homology class, and this is a contradiction which can be seen as follows. If $u^2$ is spherical, then the $4n$-skeleton of $\Omega P^{2n+2}(2)$ is homotopy equivalent to $P^{2n+1}(2) \vee S^{4n}$. On the other hand, it is easy to see that the attaching map of the $4n$-cell in $\Omega P^{2n+2}(2)$ is given by the Whitehead square $[i_{2n}, i_{2n}]$ which is nontrivial as $n>1$, whence $P^{2n+1} \cup_{[i_{2n}, i_{2n}]} e^{4n} \not\simeq P^{2n+1}(2) \vee S^{4n}$.

\ref{statement1} is equivalent to \ref{statement4}: Since the Whitehead product $[i'_{2n-1}, j'_{2n}] \in \pi_{4n-2}(V_{2n+1,2}; \mathbb{Z}/2\mathbb{Z})$ is the obstruction to extending $i'_{2n-1} \vee j'_{2n}:S^{2n-1} \vee P^{2n}(2) \to V_{2n+1,2}$ to $S^{2n-1} \times P^{2n}(2)$, this follows immediately from Proposition \ref{Whitehead products}. 

As described in \cite{S Indecomposability}, applying the Hopf construction to a map $f: S^{2n-1}\times P^{2n}(2) \to V_{2n+1,2}$ as in \ref{statement4} yields a map $H(f):P^{4n}(2)\to \Sigma V_{2n+1,2}$ with $Sq^{2n}$ acting nontrivially on $H^{2n}(C_{H(f)})$. Since $\Sigma^2 V_{2n+1,2} \simeq P^{2n+2}(2)\vee S^{4n+1}$, composing the suspension of the Hopf construction $H(f)$ with a retract $\Sigma^2 V_{2n+1,2} \to P^{2n+2}(2)$ defines a map $g:P^{4n+1}(2) \to P^{2n+2}(2)$ with $Sq^{2n}$ acting nontrivially on $H^{2n+1}(C_g)$, so \ref{statement4} implies \ref{statement3}.

By the proof of \cite[Theorem 3.1]{S Indecomposability}, \ref{statement3} implies \ref{statement5}, and \ref{statement5} implies \ref{statement6}. %That 5 implies 6 was originally proved in [CCPS] and reproduced in [Cohen].
The triviality of the Whitehead product $[i'_{2n-1}, j'_{2n}] \in \pi_{4n-2}(V_{2n+1,2}; \mathbb{Z}/2\mathbb{Z})$ when $n=1, 2$ or $4$ is implied by \cite[Theorem 2.1]{S Indecomposability}, for example, and Proposition \ref{Whitehead products} implies $[i'_{15}, j'_{16}] \in \pi_{30}(V_{17,2}; \mathbb{Z}/2\mathbb{Z})$ is trivial as well since $[i_{16}, i_{16}] \in \pi_{31}(P^{17}(2))$ is divisible by $2$ by \cite[Lemma 3.10]{MSh}. Thus \ref{statement6} implies \ref{statement4}.
\end{proof}

\section{A loop space decomposition of $J_3(S^2)$} 

In this section, we consider some relations between the fibre bundle $S^{4n-1} \to V_{4n+1,2} \to S^{4n}$ defined by projection onto the first vector of an orthonormal $2$-frame in $\mathbb{R}^{4n+1}$ (equivalently, the unit tangent bundle over $S^{4n}$) and the fibration $BW_n \to \Omega^2S^{4n+1}\{2\} \to W_{2n}$ of Lemma \ref{fibration}. Letting $\partial: \Omega S^{4n} \to S^{4n-1}$ denote the connecting map of the first fibration, we will show that there is a morphism of homotopy fibrations
\begin{equation} \label{diagram3}
\begin{gathered}
\xymatrix{
\Omega^2S^{4n} \ar[r]^-{\Omega\partial} \ar[d] & \Omega S^{4n-1} \ar[r] \ar[d] & \Omega V_{4n+1,2} \ar[d] \\
\Omega W_{2n} \ar[r] & BW_n \ar[r] & \Omega^2S^{4n+1}\{2\}
}
\end{gathered}
\end{equation}
from which it will follow that for $n=1,2$ or $4$, $\Omega\partial$ lifts through $\Omega\phi_n:\Omega^3S^{4n+1} \to \Omega S^{4n-1}$. If this lift can be chosen to be $\Omega^2 E$, then it follows that there is a homotopy pullback diagram
\begin{equation} \label{diagram4}
\begin{gathered}
\xymatrix{
\Omega^2V_{4n+1,2} \ar[r] \ar[d] & \Omega^2S^{4n} \ar[r]^{\Omega\partial} \ar[d]^{\Omega^2E} & \Omega S^{4n-1} \ar@{=}[d] \\
W_n \ar[r]^-{\Omega j} \ar[d] & \Omega^3S^{4n+1} \ar[r]^{\Omega\phi_n} \ar[d]^{\Omega^2H} & \Omega S^{4n-1} \\
\Omega^3S^{8n+1} \ar@{=}[r] & \Omega^3S^{8n+1}
}
\end{gathered}
\end{equation}
which identifies $\Omega^2V_{4n+1,2}$ with $\Omega M_3(n)$ where $\{M_k(n)\}_{k\ge 1}$ is the filtration of $BW_n$ studied in \cite{GTZ} beginning with the familiar spaces $M_1(n) \simeq\Omega S^{4n-1}$ and $M_2(n) \simeq S^{4n-1}\{\underline{2}\}$. (Spaces are localized at an odd prime throughout \cite{GTZ} but the construction of the filtration works in the same way for $p=2$.) We verify this (and deloop it) for $n=1$ since it leads to an interesting loop space decomposition which gives isomorphisms $\pi_k(V_{5,2}) \cong \pi_k(J_3(S^2))$ for all $k\ge 3$.

In his factorization of the $4^{th}$-power map on $\Omega^2S^{2n+1}$ through the double suspension, Theriault constructs in \cite{T 2-primary} a space $A$ and a map $\overline{E}: A\to \Omega S^{2n+1}\{2\}$ with the following properties:
\begin{enumerate}[label=(\alph*),ref=(\alph*)]
\item $H_\ast(A) \cong \Lambda(x_{2n-1},x_{2n})$ with Bockstein $\beta x_{2n}=x_{2n-1}$;
\item $\overline{E}$ induces a monomorphism in homology;
\item There is a homotopy fibration $S^{2n-1} \to A \to S^{2n}$ and a homotopy fibration diagram
\[
\xymatrix{
S^{2n-1} \ar[r] \ar[d]^{E^2} & A \ar[r] \ar[d]^{\overline{E}} & S^{2n} \ar[d]^E \\
\Omega^2S^{2n+1} \ar[r] & \Omega S^{2n+1}\{2\} \ar[r] & \Omega S^{2n+1}.
}
\]
\end{enumerate}
Noting that the homology of $A$ is isomorphic to the homology of the unit tangent bundle $\tau(S^{2n})$ as a coalgebra over the Steenrod algebra, Theriault raises the question of whether $A$ is homotopy equivalent to $\tau(S^{2n}) = V_{2n+1,2}$. Our next proposition shows this is true for any space $A$ with the properties above.

\begin{proposition}
There is a homotopy equivalence $A\simeq V_{2n+1,2}$.
\end{proposition}

\begin{proof}
First we show that $A$ splits stably as $P^{2n}\vee S^{4n-1}$. As in \cite{T 2-primary}, let $Y$ denote the $(4n-1)$-skeleton of $\Omega S^{2n+1}\{2\}$. Consider the homotopy fibration 
\[ \Omega S^{2n+1}\{2\} \xrightarrow{\mathmakebox[0.5cm]{}} \Omega S^{2n+1} \xrightarrow{\mathmakebox[0.5cm]{2}} \Omega S^{2n+1} \]
and recall that $H_\ast(\Omega S^{2n+1}\{2\}) \cong H_\ast(\Omega S^{2n+1}) \otimes H_\ast(\Omega^2 S^{2n+1})$.
Restricting the fibre inclusion to $Y$ and suspending once we obtain a homotopy commutative diagram
\[
\xymatrix{
 & & S^{2n+1} \ar[d] \ar[r]^{\underline{2}} & S^{2n+1} \ar[d] \\
\Sigma Y \ar[urr]^\ell \ar[r] & \Sigma\Omega S^{2n+1}\{2\} \ar[r] & \Sigma\Omega S^{2n+1} \ar[r]^{\Sigma 2} & \Sigma\Omega S^{2n+1}
}
\]
where $\underline{2}$ is the degree $2$ map, the vertical maps are inclusions of the bottom cell of $\Sigma\Omega S^{2n+1}$ and a lift $\ell$ inducing an isomorphism in $H_{2n+1}(\;)$ exists since $\Sigma Y$ is a $4n$-dimensional complex and $\mathrm{sk}_{4n}(\Sigma\Omega S^{2n+1}) = S^{2n+1}$. It follows from the James splitting $\Sigma\Omega S^{2n+1} \simeq \bigvee_{i=1}^\infty S^{2ni+1}$ and the commutativity of the diagram that $\underline{2} \circ \ell$ is nullhomotopic, so in particular $\Sigma\ell$ lifts to the fibre $S^{2n+2}\{\underline{2}\}$ of the degree $2$ map on $S^{2n+2}$. Since $H_\ast(S^{2n+2}\{\underline{2}\}) \cong \mathbb{Z}/2\mathbb{Z}[u_{2n+1}] \otimes \Lambda(v_{2n+2})$ with $\beta v_{2n+2}=u_{2n+1}$, this implies $\Sigma\ell$ factors through a map $r:\Sigma^2Y \to P^{2n+2}(2)$ which is an epimorphism in homology by naturality of the Bockstein, and hence $P^{2n+2}(2)$ is a retract of $\Sigma^2Y$. (Alternatively, $r$ can be obtained by suspending a lift $\Sigma Y \to S^{2n+1}\{\underline{2}\}$ of $\ell$ and using the well-known fact that $\Sigma S^{2n+1}\{\underline{2}\}$ splits as a wedge of Moore spaces.) Now since $\overline{E}:A\to \Omega S^{2n+1}\{2\}$ factors through $Y$ and induces a monomorphism in homology, composing $\Sigma^2A \to \Sigma^2Y$ with the retraction $r$ shows that $\Sigma^2A \simeq \Sigma^2(P^{2n}(2)\vee S^{4n-1})$.

Next, let $E^\infty: A \to QA$ denote the stabilization map and let $F$ denote the homotopy fibre of a map $g: QP^{2n}(2) \to K(\mathbb{Z}/2\mathbb{Z}, 4n-2)$ representing the mod-$2$ cohomology class $u_{2n-1}^2 \in H^{4n-2}(QP^{2n}(2))$. A homology calculation shows that the $(4n-1)$-skeleton of $F$ is a three-cell complex with homology isomorphic to $\Lambda(x_{2n-1},x_{2n})$ as a coalgebra. The splitting $\Sigma^2A \simeq \Sigma^2(P^{2n}(2)\vee S^{4n-1})$ gives rise to a map $\pi_1:QA \simeq QP^{2n}(2) \times QS^{4n-1} \to QP^{2n}(2)$ inducing isomorphisms in $H_{2n-1}(\;)$ and $H_{2n}(\;)$, and since the composite $g \circ \pi_1 \circ E^\infty: A \to K(\mathbb{Z}/2\mathbb{Z}, 4n-2)$ is nullhomotopic, there is a lift $A \to F$ inducing isomorphisms in $H_{2n-1}(\;)$ and $H_{2n}(\;)$. The coalgebra structure of $H_\ast(A)$ then implies this lift is a $(4n-1)$-equivalence and the result follows as $V_{2n+1,2}$ can similarly be seen to be homotopy equivalent to the $(4n-1)$-skeleton of $F$.
\end{proof}

The homotopy commutative diagram \eqref{diagram3} is now obtained by noting that the composite $\Omega S^{4n-1} \xrightarrow{} \Omega V_{4n+1,2} \xrightarrow{\Omega\overline{E}} \Omega^2S^{4n+1}\{2\}$ is homotopic to $\Omega S^{4n-1} \xrightarrow{\Omega E^2} \Omega^3S^{4n+1} \to \Omega^2S^{4n+1}\{2\}$, which in turn is homotopic  to a composite $\Omega S^{4n-1} \to BW_n \to \Omega^2S^{4n+1}\{2\}$ since by Theorem \ref{Richter} there is a homotopy fibration diagram
\[
\xymatrix{ 
\Omega S^{4n-1} \ar[d]^{\Omega E^2} \ar[r] & BW_n \ar[d] \ar[r]^j & \Omega^2S^{4n+1} \ar@{=}[d] \ar[r]^{\phi_n} & S^{4n-1} \ar[d]^{E^2} \\
\Omega^3S^{4n+1} \ar[r] & \Omega^2S^{4n+1}\{2\} \ar[r] & \Omega^2S^{4n+1} \ar[r]^2 & \Omega^2S^{4n+1}.
}
\]

Specializing to the case $n=1$, the proof of Proposition \ref{J_3(S^2)} will show that $\Omega V_{5,2}$ fits in a delooping of diagram \eqref{diagram4} and hence that $\Omega V_{5,2} \simeq M_3(1).$ We will need the following cohomological characterization of $V_{5,2}$.

\begin{lemma} \label{Stiefel}
Let $E$ be the total space of a fibration $S^3 \to E \to S^4$. If $E$ has integral cohomology group $H^4(E; \mathbb{Z})=\mathbb{Z}/2\mathbb{Z}$ and mod-$2$ cohomology ring $H^\ast(E)$ an exterior algebra $\Lambda(u,v)$ with $|u|=3$ and $|v|=4$, then $E$ is homotopy equivalent to the Stiefel manifold $V_{5,2}$.
\end{lemma}

\begin{proof}
As shown in \cite[Theorem 5.8]{W}, the top row of the homotopy pullback diagram
\[
\xymatrix{
X^4 \ar[r] \ar[d] & P^4(2) \ar[r] \ar[d]^q & BS^3 \ar@{=}[d] \\
S^7 \ar[r]^{\nu} & S^4 \ar[r] & BS^3
}
\]
induces a split short exact sequence
\[ 0 \xrightarrow{\mathmakebox[0.5cm]{}} \mathbb{Z}/4\mathbb{Z} \xrightarrow{\mathmakebox[0.5cm]{}} \pi_6(P^4(2)) \xrightarrow{\mathmakebox[0.5cm]{}} \pi_5(S^3) \xrightarrow{\mathmakebox[0.5cm]{}} 0 \]
from which it follows that $\pi_6(P^4(2)) = \mathbb{Z}/4\mathbb{Z}\{\lambda\} \oplus \mathbb{Z}/2\mathbb{Z}\{\widetilde{\eta}_3^2\}$ where $\lambda$ is the attaching map of the top cell of $V_{5,2}$ and $\widetilde{\eta}_3^2$ maps to the generator $\eta_3^2$ of $\pi_5(S^3)$. It follows from the cohomological assumptions that $E \simeq P^4(2)\cup_f e^7$, where $f=a\lambda + b\widetilde{\eta}_3^2$ for some $a \in \mathbb{Z}/4\mathbb{Z}$, $b \in \mathbb{Z}/2\mathbb{Z}$, and that $H_\ast(\Omega E)$ is isomorphic to a polynomial algebra $\mathbb{Z}/2\mathbb{Z}[u_2,v_3]$. Since the looped inclusion $\Omega P^4(2) \to \Omega E$ induces the abelianization map $T(u_2,v_3) \to \mathbb{Z}/2\mathbb{Z}[u_2,v_3]$ in homology, it is easy to see that the adjoint $f':S^5 \to \Omega P^4(2)$ of $f$ has Hurewicz image $[u_2, v_3]=u_2\otimes v_3 + v_3\otimes u_2$ and hence $f$ is not divisible by $2$. Moreover, since $E$ is an $S^3$-fibration over $S^4$, the pinch map $q:P^4(2)\to S^4$ must extend over $E$. This implies the composite $S^6 \xrightarrow{f} P^4(2) \xrightarrow{q} S^4$ is nullhomotopic and therefore $b=0$ by the commutativity of the diagram above. It now follows that $f=\pm \lambda$ which implies $E \simeq V_{5,2}$.
\end{proof}

\begin{proposition} \label{J_3(S^2)}
There is a homotopy fibration
$$V_{5,2} \xrightarrow{\mathmakebox[0.5cm]{}} J_3(S^2) \xrightarrow{\mathmakebox[0.5cm]{}} K(\mathbb{Z}, 2)$$
which is split after looping.
\end{proposition}

\begin{proof}
Let $h$ denote the composite $\Omega S^3\langle 3\rangle \xrightarrow{} \Omega S^3 \xrightarrow{H} \Omega S^5$ and consider the pullback
\[
\xymatrix{
P \ar[r] \ar[d] & S^4 \ar[d]^E \\
\Omega S^3\langle 3\rangle \ar[r]^-h & \Omega S^5.
}
\]
Since $h$ has homotopy fibre $S^3$, so does the map $P\to S^4$. Next, observe that $P$ is the homotopy fibre of the composite $\Omega S^3\langle 3\rangle \xrightarrow{h} \Omega S^5 \xrightarrow{H} \Omega S^9$ and since $\Omega S^9$ is $7$-connected, the inclusion of the $7$-skeleton of $\Omega S^3\langle 3\rangle$ lifts to a map $\mathrm{sk}_7(\Omega S^3\langle 3\rangle) \to P$. Recalling that $H^4(\Omega S^3\langle 3\rangle; \mathbb{Z}) \cong \mathbb{Z}/2\mathbb{Z}$ and $H_\ast(\Omega S^3\langle 3\rangle) \cong \Lambda(u_3) \otimes \mathbb{Z}/2\mathbb{Z}[v_4]$ with generators in degrees $|u_3|=3$ and $|v_4|=4$, it follows that this lift must be a homology isomorphism and hence a homotopy equivalence. So $P$ is homotopy equivalent to the total space of a fibration satisfying the hypotheses of Lemma \ref{Stiefel} and there is a homotopy equivalence $P\simeq V_{5,2}$.

It is well known that the iterated composite of the $p^{th}$ James--Hopf invariant $H^{\circ k}: \Omega S^{2n+1} \to \Omega S^{2np^k+1}$ has homotopy fibre $J_{p^k-1}(S^{2n})$, the $(p^k-1)^{st}$ stage of the James construction on $S^{2n}$. The argument above identifies $V_{5,2}$ with the homotopy fibre of the composite
\[ \Omega S^3\langle 3\rangle \xrightarrow{\mathmakebox[0.5cm]{}} \Omega S^3 \xrightarrow{\mathmakebox[0.5cm]{H}} \Omega S^5 \xrightarrow{\mathmakebox[0.5cm]{H}} \Omega S^9, \]
so there is a homotopy pullback diagram
\[
\xymatrix{
V_{5,2} \ar[r] \ar[d] & J_3(S^2) \ar[r] \ar[d] & K(\mathbb{Z}, 2) \ar@{=}[d] \\
\Omega S^3\langle 3\rangle \ar[r] \ar[d]^{H\circ h} & \Omega S^3 \ar[r] \ar[d]^{H^{\circ 2}} & K(\mathbb{Z}, 2) \\
\Omega S^9 \ar@{=}[r] & \Omega S^9
}
\]
where the maps into $K(\mathbb{Z}, 2)$ represent generators of $H^2(J_3(S^2); \mathbb{Z}) \cong \mathbb{Z}$ and $H^2(\Omega S^3; \mathbb{Z}) \cong \mathbb{Z}$. To see that the homotopy fibration along the top row splits after looping, note that the connecting map $\Omega K(\mathbb{Z}, 2)=S^1 \to V_{5,2}$ is nullhomotopic since $V_{5,2}$ is simply-connected. Therefore the looped projection map $\Omega J_3(S^2) \to S^1$ has a right homotopy inverse producing a splitting $\Omega J_3(S^2) \simeq S^1 \times \Omega V_{5,2}$.
\end{proof}

\begin{corollary}
$\pi_k(J_3(S^2)) \cong \pi_k(V_{5,2})$ for all $k\ge 3$.
\end{corollary}

%%%%%%%%%%%%%%%%%%%%%%%%%%%%%%%

\end{document}